\documentclass[12pt]{amsart}
\usepackage{hyperref}
\usepackage{enumerate}

\begin{document}
\numberwithin{equation}{section}

\def\1#1{\overline{#1}}
\def\2#1{\widetilde{#1}}
\def\3#1{\widehat{#1}}
\def\4#1{\mathbb{#1}}
\def\5#1{\frak{#1}}
\def\6#1{{\mathcal{#1}}}

\def\C{{\4C}}
\def\R{{\4R}}
\def\N{{\4N}}
\def\Z{{\4Z}}

\title[Triangular resolutions and effectiveness]
{Triangular resolutions and effectiveness for holomorphic subelliptic multipliers}
\author[S.-Y. Kim \& D. Zaitsev]{Sung-Yeon Kim and Dmitri Zaitsev}
\address{S.-Y. Kim: Center for Mathematical Challenges, Korea Institute for Advanced Study, 85 Hoegiro, Dongdaemun-gu
Seoul, Korea }
\email{sykim8787@kias.re.kr}
\address{D. Zaitsev: School of Mathematics, Trinity College Dublin, Dublin 2, Ireland}
\email{zaitsev@maths.tcd.ie}
\maketitle

\def\Label#1{\label{#1}}


\def\cn{{\C^n}}
\def\cnn{{\C^{n'}}}
\def\ocn{\2{\C^n}}
\def\ocnn{\2{\C^{n'}}}


\def\dist{{\rm dist}}
\def\const{{\rm const}}
\def\rk{{\rm rank\,}}
\def\id{{\sf id}}
\def\aut{{\sf aut}}
\def\Aut{{\sf Aut}}
\def\CR{{\rm CR}}
\def\GL{{\sf GL}}
\def\Re{{\sf Re}\,}
\def\Im{{\sf Im}\,}
\def\span{\text{\rm span}}
\def\mult{\text{\rm mult\,}}
\def\reg{\text{\rm reg\,}}
\def\ord{\text{\rm ord\,}}
\def\hot{\text{\rm HOT\,}}

\def\codim{{\rm codim}}
\def\crd{\dim_{{\rm CR}}}
\def\crc{{\rm codim_{CR}}}

\def\eps{\varepsilon}
\def\d{\partial}
\def\a{\alpha}
\def\b{\beta}
\def\g{\gamma}
\def\G{\Gamma}
\def\D{\Delta}
\def\Om{\Omega}
\def\k{\kappa}
\def\l{\lambda}
\def\L{\Lambda}
\def\z{{\bar z}}
\def\w{{\bar w}}
\def\Z{{\1Z}}
\def\t{\tau}
\def\th{\theta}

\emergencystretch15pt
\frenchspacing

\newtheorem{Thm}{Theorem}[section]
\newtheorem{Cor}[Thm]{Corollary}
\newtheorem{Pro}[Thm]{Proposition}
\newtheorem{Lem}[Thm]{Lemma}
\newtheorem{Prob}[Thm]{Problem}

\theoremstyle{definition}\newtheorem{Def}[Thm]{Definition}

\theoremstyle{remark}
\newtheorem{Rem}[Thm]{Remark}
\newtheorem{Exa}[Thm]{Example}
\newtheorem{Exs}[Thm]{Examples}

\def\bl{\begin{Lem}}
\def\el{\end{Lem}}
\def\bp{\begin{Pro}}
\def\ep{\end{Pro}}
\def\bt{\begin{Thm}}
\def\et{\end{Thm}}
\def\bc{\begin{Cor}}
\def\ec{\end{Cor}}
\def\bd{\begin{Def}}
\def\ed{\end{Def}}
\def\br{\begin{Rem}}
\def\er{\end{Rem}}
\def\be{\begin{Exa}}
\def\ee{\end{Exa}}
\def\bpf{\begin{proof}}
\def\epf{\end{proof}}
\def\ben{\begin{enumerate}}
\def\een{\end{enumerate}}
\def\beq{\begin{equation}}
\def\eeq{\end{equation}}

\begin{abstract}
A solution to the effectiveness problem
in Kohn's algorithm for generating subelliptic multipliers
is provided for domains that include
those given by sums of squares of holomorphic functions
(also including infinite sums).
These domains 
are of particular interest
due to their relation with 
complex and algebraic geometry
and in particular 
seem to include all previously known cases.
Furthermore, combined with a recent result of M.~Fassina~\cite{Fa20},
our effectiveness method allows
to establish effective subelliptic estimates
for more general classes of domains.

Our main new tool,
 {\em a triangular resolution},
is the construction
of subelliptic
 multipliers decomposable as $Q\circ\G$,
where $\G$ is constructed from pre-multipliers
and $Q$ is part of a triangular system.
The effectiveness is proved
via a sequence of newly proposed procedures,
called here {\em meta-procedures},
built on top of the Kohn's procedures,
where the order of subellipticity
can be effectively tracked.
Important sources of inspiration
are algebraic geometric techniques by Y.-T. Siu \cite{S10, S17}
and procedures for triangular systems 
by D.W. Catlin and J.P. D'Angelo 
\cite{D95, CD10}.

The proposed procedures are purely algebraic
and as such
can also be of interest
for geometric and computational problems involving 
Jacobian determinants,
such as resolving singularities of holomorphic maps.


\end{abstract}

\section{Introduction}
The Kohn's technique of subelliptic multipliers \cite{K79}
is one of the few known general techniques
connecting a priori estimates 
for systems of partial differential equations
(subelliptic estimates for the $\bar\d$-Neumann problem \cite[Definition~1.11]{K79})
with tools from commutative algebra and algebraic geometry.
Kohn's key innovations include
refining subelliptic estimates 
by introducing multiplier functions \cite[\S2.8]{S17}, 
proposing purely algebraic procedures
to generate new multipliers,
and developing an algorithm to apply these procedures
to obtain subelliptic estimates
for real-analytic pseudoconvex domains of finite D'Angelo type.
The reader is referred to
\cite{D95, DK99, S01, S02, K04, S05, S07, Ch06, S09, CD10, S10, S17}
for more extensive details on subelliptic multipliers
and Siu's accounts \cite{S07, S09, S17}
on their broad role and relation with 
other multipliers arising in complex and algebraic geometry.

On the other hand, the {\em question of effectiveness,
i.e. control of the Sobolev exponent} of multipliers in Kohn's algorithm
has remained an open problem since Kohn's work of 1979
(apart from the complex dimension $2$, see \cite[\S8]{K79},
where it is based on fundamental results
by H\"ormander \cite{H65} and Rothschild-Stein \cite{RS76}).
In higher dimension, the situation is much less understood,
in fact, examples of \cite{Heier} (\S1.1 in the preprint version)
and \cite[Proposition 4.4]{CD10} in dimension $3$
illustrate a lack of such control,
see also \cite[\S4.1]{S17} for
a detailed explanation of this important phenomenon.

To tackle the effectiveness,
Siu \cite{S10, S17} 
 introduced algebraic geometric techniques
to obtain the effectiveness 
 in the important case of 
{\em special domains} of finite type
in dimension $3$,
with further indications how to proceed in
the more general cases of
special domains in higher dimension,
and outlining a program to treat 
the more general real-analytic and smooth cases.
A different effective procedure
in Kohn's algorithm
was given by D'Angelo~\cite{D95}
and Catlin-D'Angelo \cite[Section~5]{CD10}
for special domains given by 
so-called {\em triangular systems}
of holomorphic functions.
In \cite{N14}  A.C.~Nicoara
proposed a construction for the termination of the Kohn algorithm in the real-analytic case with an indication of the ingredients needed for the effectivity.
More recently, the authors of this article 
established another effective procedure in dimension $3$
by means of a new tool of jet-vanishing orders \cite{KZ18},
where the reader is referred for further discussion and references.

After receiving a preprint of our paper, Y.-T. Siu told us about his unpublished proof of the effective termination of KohnÕs algorithm for special domains of arbitrary dimension by using descending induction on the dimension of the subvariety defined by effectively constructed multipliers and the techniques of multiplicities of Jacobian determinants and fiberwise differentiation without the complication of considering generalizations of the result of 
Skoda-Brian\c con.

\subsection{Main results}
In this paper we establish, in particular,
the effectiveness for special domain of arbitrary dimension
by means of the proposed new geometric tools
of {\em triangular resolutions} and
{\em effective meta-procedures}
that are
largely inspired by the
algebraic geometric techniques by
\hbox{Y.-T.~Siu}~\cite{S10, S17}
and procedures for triangular systems by
D.W.~Catlin and J.P.~D'Angelo \cite{D95, CD10}.

Recall first the definition of special domains
 \cite[\S7]{K79}, \cite[\S2.8]{S17}:

\bd
A {\em special domain} in $\C^{n+1}$
is one defined locally near each boundary 
point by
\beq\Label{om}
	\Re(z_{n+1})+\sum_{j=1}^N |F_j(z_1,\ldots, z_n)|^2<0
\eeq
where $F_1,\ldots,F_N$ are holomorphic functions.
\ed
 
Next recall \cite[\S7]{K79} that for special domains,
Kohn's multiplier generation procedures 
can be formulated purely in terms of holomorphic functions,
starting from the given set 
$S=\{F_1, \ldots, F_N\}$ of the functions in \eqref{om}.
We also adopt Siu's terminology \cite{S10,S17},
calling functions in $S$ {\em pre-multipliers} 
and the Sobolev exponent 
used in the
modified subelliptic estimate
with factor $f$,
{\em the order of subellipticity} 
or just {\em the order} of the multiplier $f$.


\bd\Label{k-hol}
For an initial set $S$ 
of {\em pre-multipliers},
the {\em Kohn's procedures} consist of:
\begin{enumerate}
\item[(P1)]
	for $0<\eps\le 1/2$ and $f_1, \ldots, f_n$ either in $S$
	or multipliers of order $\ge\eps$,
	it follows that
	the Jacobian determinant
	$$
		g
		:=
		\det\left(\frac{\d f_i}{\d z_j}\right)
		= 
		\frac{\d(f_1,\ldots,f_n)}{\d(z_1,\ldots,z_n)}
	$$
	is a multiplier of order $\ge \eps/2$;
	
\item[(P2)] 
	for $0<\eps<1$, $k,r\in\N_{\ge 1}$,  
	$f_1, \ldots, f_k$ multipliers or order $\ge\eps$,
	and
	$g$ a holomorphic function (germ) with $g^r\in (f_1, \ldots, f_k)$,
	it follows that $g$ is a multiplier of order $\ge \eps/r$.
\end{enumerate}
\ed

Note that it is (P2) that is responsible for the lack of effectiveness,
as it is a priori not clear what root order $r$ needs to be used.
To formulate our first result recall 
from D'Angelo
\cite{D79, D82, D93}
and Siu \cite{S10, S17}
that the order of finite type
can be estimated on both sides
in terms of other quantities
(see also 
Boas-Straube \cite{BS92},
Baouendi-Huang-Rothschild \cite{BHR96}, 
Fu-Isaev-Krantz \cite{FIK96},
Fornaess-Lee-Zhang \cite{FLZ14},
Brinzanescu-Nicoara \cite{BN15, BN19},
McNeal-Mernik  \cite{MM17},
 D'Angelo \cite{D17}, 
Fassina \cite{Fa19, Fa20}, 
Huang-Yin \cite{HY19},
and the second author \cite{Z19}
for relations with other invariants).
In this paper we use the {\em multiplicity}:

\bd
The {\em multiplicity} of 
an ideal $I\subset\6O_{n,p}$ 
in the ring of germs at $p$ 
of holomorphic functions in $\cn$
is the dimension of the quotient
$\dim(\6O_{n,p}/I)\le \infty$.
The {\em multiplicity} of
a subset $S\subset \6O_{n,p}$
is the multiplicity of its generated ideal.
The {\em multiplicity} of
a domain \eqref{om}
at a boundary point $p$ is
the multiplicity of 
the set of all germs at $p$ of
the functions $F_j - F_j(p)$,
 $1\le j\le N$.
\ed

To emphasize its algebraic nature,
we formulate our first result 
purely in terms of the
Kohn's procedures (P1) and (P2):

\bt\Label{p12}
For every initial subset $S\subset \6O_{n,p}$ of finite multiplicity $\le \nu$,
there exists an effectively computable sequence 
of germs at $p$ of holomorphic functions
$f_1,\ldots, f_m$, where $f_m=1$ and each 
$f_j$ is either in $S$ or is obtained by applying 
to $(f_1,\ldots,f_{j-1})$
one of the Kohn's procedures (P1) or (P2).
Furthermore, the number of steps $m$ and the root orders in (P2)
are effectively bounded by functions depending only on $(n,\nu)$.
\et

In the next result we apply
Theorem~\ref{p12}
to obtain effective subelliptic estimates,
with additional explicit bound
for the order of subellipticity:

\bt
There exists a positive function 
$\eps\colon \N_{>0} \times \N_{>0} \to \R_{>0}$
such that for any domain 
\eqref{om}
of finite multiplicity 
$\le \nu$ at a boundary point $p$,
a subelliptic estimate holds at $p$
with effectively bounded
order of subellipticity $\ge \eps(n,\nu)$.
In fact, one can take
\beq\Label{e-est}
	\eps(n,\nu)
	=
	\left(
			4
		(2n+2)^{2(n\nu)^{ (3n)^{n+1} }}
	\right)^{-1}.
\eeq
\et

\br
Since the multiplicity $\nu$ of $I$ satisfies
$\nu \le t^n$
where $t$ is the type
by a result of D'Angelo 
\cite[Theorem~2.7]{D82},
an effective bound in terms of the type
can be obtained by
substituting $t^n$ for $\nu$
in \eqref{e-est}.
\er

\subsection{More general domains and applications}
Due to the algebraic nature of 
Theorem~\ref{p12} referring only to Kohn's procedures
rather than to domains,
the conclusion can be applied for more general domains,
in conjunction with other effective procedures
 to obtain the initial set $S$ in Definition~\ref{k-hol}. 
Consider general domains given locally 
near a boundary point $p$ by
\beq\Label{psh-domain}
	\rho(z,\bar z) < 0,
\eeq
where $d\rho\ne 0$.
Let $S \subset \6O_{n,p}$
be the set of
holomorphic function germs $F$ such that
\beq\Label{S}
	\Big|
		\sum_j 
		\frac{\d F}{\d z_j}
		v_j
	\Big|^2
	\le
	c
	\left(
	\sum_{j,k}
	\frac{
		\d^2 \rho
	}{
		\d z_j \d \bar z_k
	} v_j \bar v_k
	+ 
	\Big|
		\sum_j 
		\frac{\d \rho}{\d z_j}
		v_j
	\Big|^2
	\right)
\eeq
for some
$ c>0$ 
 and all 
 $v_1, \ldots, v_n$.
It follows from 
Fassina \cite[Remark~4.5]{Fa20}
 that $S$
consists of {\em pre-multipliers
in the sense of Siu} \cite[\S2.9]{S17}
that can be used in Kohn's procedures.
(In particular, if 
$\rho=\Re z_{n+1} + \phi(z_1,\ldots,z_n,\bar z_1,\ldots,\bar z_n)$
where $\phi$ is an infinite sum
of squares of holomorphic functions $F_j$, 
each of $F_j$ is contained in $S$.)
Hence we obtain the following application of Theorem~\ref{p12}:

\bc\Label{non-special}
With $\eps(n,\nu)$ given by \eqref{e-est},
for a smooth pseudoconvex domain
\eqref{psh-domain},
assume that
the multiplicity of
$S$
is $\le \nu$.
Then 
a subelliptic estimate holds at $p$
with order $\ge \eps(n,\nu)$.
In particular, an effective subelliptic estimate
holds for domains \eqref{e-est}
 of finite type
with $\phi$ an infinite sum
of squares of holomorphic functions.
\ec


We conclude this subsection
by mentioning applications of 
(1) effective subelliptic estimates
to local regularity with effective gain \cite{KN65},
 effective lower bound on the Bergman metric \cite{M92}
and control in the construction of peak functions
\cite{FM94}
and (2) effective Kohn's algorithm
(such as given by Theorem~\ref{p12})
to the construction of``bumping''
functions, Kobayashi metric estimates
and H\"older regularity of proper 
holomorphic maps \cite{DF79}.

\subsection{Triangular resolutions and effective meta-procedures}
In this section we introduce our main tools for establishing effectiveness.
Recall that the crucial lack of effectiveness in  (P2) (see Definition~\ref{k-hol}) is due to the order of the generated multiplier depending on the root order that can happen to be arbitrarily large.

To quantify this phenomenon, we call a procedure
{\em effective} if the order of the new multiplier
can be effectively estimated in terms of 
a quantity associated to the data that we call {\em a complexity}.
We don't seek complexities of individual multipliers 
but rather of their {\em finite tuples and tuples of their ideals,
or more precisely, their filtrations}.
That is, we define,
for a holomorphic map germ
given by a tuple of pre-multipliers,
and a filtration of ideals of multipliers,
the notion of triangular resolution as follows:

\bd
A {\em triangular resolution of length $k\ge 1$}
and multi-order $(\mu_1,\ldots,\mu_k)\in \N^k$
of a pair $(\G, \6I)$,
where 
$\G\colon (\cn,0)\to (\C^n,0)$
is a holomorphic map germ
and
$I_1\subset \ldots \subset I_k \subset \6O_{n,0}$
 a filtration $\6I$ of ideals,
is 
a system of holomorphic function germs
$(h_1,\ldots,h_k)$
satisfying
$$
	h_j = h_j(w_j, \ldots, w_n),
	\quad
	h_j\circ \G \in I_j,
	\quad
	\ord_{w_j} h_j = \mu_j,
	\quad
	1\le j\le k.
$$
\ed

\br\Label{proj}
Triangular resolutions can be 
given an
equivalent
coordinate-free description by means of pairs $(s,h)$,
where $s=(s_1,\ldots,s_k)$ is a sequence of holomorphic submersions 
$$
	(\C^n,0) 
	\mathop{\to}^{s_1} 
	(\C^{n-1},0)
	\mathop{\to}^{s_2} 
		\ldots 
	\mathop{\to}^{s_{k}} 
	(\C^{n-k},0)
$$
and 
$h=(h_1,\ldots, h_k)$ 
a $k$-tuple of
holomorphic function germs 
satisfying
$$
	h_j\in \6O_{\C^{m-j+1},0},
	\quad
	h_j\circ s_{j-1}\circ \ldots \circ s_1\circ \G \in I_j,
	\quad
	\mult(h_j, s_j) = \mu_j,
	\quad
	j=1,\ldots, k.
$$
\er

It is the {\em multi-order} of a triangular resolution
that plays the role of the complexity as mentioned above
in our effective procedures,
that we call here {\em meta-procedures}
in order to emphasize
that each of them is a constructed as a sequence 
of Kohn's original procedures (P1) and (P2).

\br
Since the multi-order depends on the resolution,
we can define
invariant polytopes $P_k = P_k(\G, \6I) \subset \R^k$
as convex hulls of 
all multi-orders arising in this way,
that can be used for a fine-grained
control of the effectiveness.
As the construction of triangular resolution
involves projections (submersions) to decreasing 
sequence of subspaces in the target space,
$P_k$ can be seen as dual to the Newton-Okounkov polytopes 
\cite{O96}, \cite{LM09, KK12},
where instead increasing sequences
of subvarieties are used to compute the vanishing orders.
\er

Our proof of the results from previous section
is based on the the following effective meta-procedures
involving triangular resolutions:

\bt\Label{meta}
For $0\le k\le n$, the following hold:
\begin{enumerate}
\item[(MP1)] (Selection of a partial Jacobian).
For any
$$
	f=(f_1,\ldots,f_k)\in (\6O_{n,0})^k,
		\quad 
	\psi=(\psi_{k+1}, \ldots, \psi_n)\in (\6O_{n,0})^{n-k},
		\quad
	\mult(f,\psi) \le \mu <\infty,
$$
there exist linear changes 
of the coordinates $z\in \C^n$
and 
of the components of $\psi$ in $\C^{n-k}$ 
such that for
the partial Jacobian determinant
$$
	J:= \frac{\d(\psi_{k+1},\ldots,\psi_n)}{\d(z_{k+1},\ldots,z_n)},
$$
the multiplicity
$
	\mult(f,J,\psi_{k+2}, \ldots, \psi_n)
$
is effectively bounded by a function depending only on 
$(n,\mu)$.

\item[(MP2)] (Selection of a triangular resolution).
For any 
$$
	f=(f_1,\ldots,f_k)\in (\6O_{n,0})^{k},
		\quad 
	\psi=(\psi_{1}, \ldots, \psi_n)\in (\6O_{n,0})^n,
$$
with
$$
	\mult(f_1,\ldots,f_j, \psi_{j+1}, \ldots, \psi_n)\le \mu < \infty,
	\quad
	0\le j\le n,
$$
there exists a triangular resolution $h=(h_1,\ldots,h_k)$ of $(\psi,\6I)$,
where $\6I$ is the filtration 
$$
	(f_1)\subset (f_1, f_2) \ldots \subset (f_1,\ldots, f_k),
$$
such that orders $\ord_{w_j}h_j$ are 
effectively bounded by functions depending only on 
$(n,\mu)$.

\item[(MP3)] (Jacobian extension in a triangular resolution).
For any 
$$
\G=(\phi,\psi)\in (\6O_{n,0})^{k}\times (\6O_{n,0})^{n-k},
$$
and filtration $\6I$ of ideals 
$I_1\subset \ldots \subset I_{k+1}\subset \6O_{n,0}$
satisfying 
$$
	I_{k+1}\subset I_k + (J),
$$ 
where $J$ is the Jacobian determinant of $\G$,
let $h=(h_1,\ldots, h_{k+1})$ be a triangular resolution
with
$$
	\ord_{z_j}h_j \le \mu <\infty,
	\quad
	1\le j\le k.
$$
Then 
$
	h_{k+1}\circ  \G
$
can be obtained
by holomorphic Kohn's procedures 
(P1) and (P2)
starting with the initial set of components of 
$\psi$
and the ideal $I_k$,
where the number of procedures and the root order in (P2)
are effectively bounded 
by a function depending only on $(n,\mu)$.
\end{enumerate}
\et

The proof for each of the statements in (MP1), (MP2) and (MP3)
will be provided in their more precise versions 
with explicit estimates
in Propositions~\ref{v}, \ref{triangle} and \ref{claim} respectively.
All three meta-procedures are subsequently combined
in Corollary~\ref{iter}
providing an iteration step in the main construction,
where one more multiplier is added to the system.
For reader's convenience, 
here is a self-contained variant of
Corollary~\ref{iter}:

\bc\Label{k}
	For any integer $\nu\ge 1$, initial system of pre-multipliers 
	$\psi_0=(\psi_{0,1}, \ldots, \psi_{0,n})$
	of finite type $ \le \nu$,
	and  $1\le k\le n$, there exist:
	\begin{enumerate}
	\item
		holomorphic coordinates $(z_1, \ldots, z_n)$ chosen among linear combinations of any given holomorphic coordinate system;
	\item
		systems of pre-multipliers
		$\psi_k = (\psi_{k,k+1}, \ldots, \psi_{k,n})$
		chosen among generic linear combination of the given ones,
		and associated maps
		$$
			\G_k(z):= (z_1, \ldots, z_k, \psi_{k,k+1}(z), \ldots, \psi_{k,n}(z));
		$$
	\item
		systems of multipliers
		$f_k = (f_{k,1}, \ldots, f_{k,k})$
		obtained
		via effective meta-procedures applied
		to $(\psi_{k-1}, f_{k-1} )$ (where $f_0$ is empty);
	\item
		decompositions
		of the form $f_{k,j} = Q_{k,j}\circ \G_k$, 
		$j=1, \ldots, k$,
		where 
		each $Q_{k,j}=Q_{k,j}(w_j, \ldots, w_n)$ is a holomorphic function 
		depending only on the last $n-j+1$ coordinates;
	\item 
		positive functions $\eps_{k,j}(n, \nu)>0$
		such that the order of subellipticity of each $f_{k,j}$ 
		is $\ge\eps_{k,j}(n, \nu)$.
	\end{enumerate}
\ec

Finally the estimates from
Theorem~\ref{p12}
are provided by
Corollary~\ref{compute}.

\br
Since all meta-procedures (MP1-3) in Theorem~\ref{meta} are purely algebraic, they can be applied to geometric problems beyond subelliptic estimates. For instance, consider a holomorphic map $\phi\colon V\to \C^m$
on a singular (not necessarily reduced) subvariety $V\subset\C^n$ 
given by a sheaf of ideals $\6I$, 
and look for canonical ways of resolving the singularity of $\phi$.
Here adding a Jacobian determinant in (P1) to the ideal sheaf $\6I$
can be regarded as a step in simplifying the singularity of $\phi$,
since in the regular case the Jacobian would be a unit.
The singularity itself can then be encoded
by trees whose edges correspond to
(MP1-3) and that are labeled by
 multi-orders of the triangular resolutions.
Such trees can be used to obtain 
more refined effective subelliptic estimates
and put in a broader context
of labeled trees encoding singularities,
see e.g.\ \cite{GGP19}.
\er

We conclude with
Section~\ref{example}
 illustrating
the use of our effective algorithm
for a concrete class of examples 
with arbitrary high order perturbations.

\medskip
{\bf Acknowledgement.}
We would like to thank Professors
J.P.~D'Angelo, X.~Huang and Y.-T.~Siu
for their interest and helpful comments.

\section{Preliminaries}
\subsection{Multiplicity and degree}
Denote by $\6O=\6O_{n,p}$
the ring of germs at a point $p$
of holomorphic functions in $\C^n$.
Since our considerations are 
for germs at a fixed point,
we shall assume $p=0$
unless specified otherwise.

Recall that an ideal
$\6I\subset\6O$
is of {\em finite type}
if $\dim\6O/\6I<\infty$,
or equivalently the (germ at $0$ of the) zero variety 
$\6V(\6I)$
is zero-dimensional at $0$.
In the latter case,
the classical {\em algebraic intersection multiplicity} of $\6I$
(see e.g.\ \cite[\S1.6, \S2.4]{Fu84})
is defined as
\beq
	\mult \6I := \dim \6O/\6I.
\eeq
Similarly, for a germ of holomorphic map
$\psi\colon(\C^n,0)\to (\C^n,0)$,
we have $\mult\psi := \mult (\psi)$,
where $(\psi)$ is the ideal generated
by the components of $\psi$,
and the quotient $\6O/(\psi)$ is the {\em local algebra of $\psi$}
(see e.g.\ \cite{AGV85}).
More generally (cf.~ \cite[\S2.4]{D93}), for every integer $0\le d < n$,
define the {\em $d$-multiplicity} by
\beq\Label{6I}
	\mult_d \6I := \min \dim  \6O/(\6I + (L_1, \ldots, L_d)),
\eeq
where the minimum is taken over
sets of $d$  linear functions $L_j$ on $\C^n$.
The same minimum is achieved
when $L_j$ are germs
of holomorphic functions with linearly independent differentials,
as can be easily shown by a change of coordinates
linearlizing the functions.
In particular, the $d$-multiplicity of an ideal is a {\em biholomorphic invariant}.
In a similar vein, given a collection 
$\phi=(\phi_1, \ldots, \phi_{n-d})\in \6O_n$,
of
$n-d$ function germs,
we write
\beq\Label{mult-d}
	\mult (\phi)
	=
	\mult (\phi_1, \ldots, \phi_{n-d}) 
	:= \min \dim  \6O/(\phi_1, \ldots, \phi_{n-d}, L_1, \ldots, L_d),
\eeq
where $L_j$ are as above.
That is, we will adopt the following convention:

\medskip

{\bf Convention.}
For every $0\le k\le n$ and a $k$-tuple 
of holomorphic function germs
$\phi_1,\ldots, \phi_k$,
their multiplicity 
$\mult(\phi_1,\ldots, \phi_k)$
is always assumed 
to be the $(n-k)$-multiplicity,
i.e.\ with $(n-k)$ generic linear functions
added to the ideal.
\medskip


Further recall that the {\em degree} 
$\deg(\psi)$
of a germ
(also called ``index" in \cite{AGV85}) 
of a finite holomorphic map
$\psi\colon(\C^n,0)\to (\C^n,0)$
is the minimum $m$ such that
$\psi$ restricts to a ramified $m$-sheeted covering
between neighborhoods of $0$ in $\C^n$.
%
Both integers are known to coincide (see e.g.\ \cite{ELT77, AGV85, D93}):

\begin{Thm}
[{\cite[\S4.3]{AGV85}}]
\Label{degree}	
Let $\psi\colon(\C^n,0)\to (\C^n,0)$ be germ of finite holomorphic map.
Then
$$
	\mult (\psi) = \deg \psi.
$$
\end{Thm}

\subsection{Semi-continuity of multiplicity and application}
The following lemma is straightforward consequence of the definition of multiplicity:
\bl\Label{semi}
Let $\psi_t \colon  (\C^m,0)\to (\C^k,0)$
be a continuous family of germs of holomorphic maps,
in the sense that all coefficients of the power series expansion of $\psi_t$
depend continuously on $t\in \R^m$.
Then $\mult(\psi_t)$ is upper semicontinuous in $t$.
\el

In the following we keep using the notation \eqref{mult-d}.

\bc\Label{ineq}
For every germs
$$
	(f,g)\colon (\C^{n+m}, 0) \to (\C^n, 0) \times (\C^m, 0),
$$
we have
$$
	\mult f \le \mult (f,g).
$$
\ec

\bpf
Consider the family $g_t := g + t(z_1, \ldots, z_m)$.
Then Lemma~\ref{semi} implies that
$$
	\mult (f, g_t) \le \mult (f,g)
$$
 for $t$ near $0$.
 Choosing $t$ such that $g_t$ is immersive at $0$,
 we can change holomorphic coordinates in $\C^{n+m}$
 to make $g_t$ linear.
 Then the desired conclusion follows from the definition of multiplicity.
\epf

\subsection{Siu's lemma on effective mixed multiplicity}
An important ingredient is the following 
consequence from Siu's lemma
on selection of linear combinations
of holomorphic functions
 for effective multiplicity
 \cite[(III.3)]{S10}
 combined with effective comparison
 of the invariants of holomorphic map germs
 \cite[(I.3-4)]{S10}
 (see also \cite[\S2.2]{D93}):
\bl[Siu]\Label{siu}
Let $0\le j\le q\le n\le N$,
and
 $f_1,\ldots, f_j, F_1, \ldots, F_N$
be holomorphic function germs in $\6O_{n,0}$ 
such that
$$
	\mu:= \mult(f_1, \ldots, f_j) <\infty,
	\quad
	\nu:=\mult(F_1,\ldots, F_N) <\infty.
$$
Then
$$
	\mult(f_1, \ldots, f_j, G_1, \ldots, G_{n-q}) 
	\le
	\mu  \nu^{n-q}
$$
holds for generic linear combinations
$G_1, \ldots, G_{n-q}$ of $F_j$'s.
\el

\section{Effective Nullstellensatz}

We shall need the following effectiveness lemma
that essentially follows the lines of Heier~\cite{Heier},
(see also \cite{LT08} for related techniques):
\begin{Lem}[Effective Nullstellensatz]\Label{radical}
Let $\phi_1,\ldots,\phi_k,f\in \mathcal{O}_{n,0}$
satisfy
\begin{equation*}
	\mu := \mult(\phi_1,\ldots,\phi_k)<\infty,
		\quad
	f\in \sqrt{(\phi_1,\ldots,\phi_k)}.
\end{equation*}
Then
$$
	f^{n\mu }\in (\phi_1,\ldots,\phi_k).
$$
\end{Lem}

\begin{proof}
Let $(\phi):=(\phi_1,\ldots,\phi_k)$ be the ideal of $\mathcal{O}_n$ generated by $\phi_1,\ldots,\phi_k$.
By the definition of multiplicity, 
for a generic choice
of linear functions $L_{k+1},\ldots, L_n$, one has
\begin{equation*}
\mu =
\text{\rm mult}(\phi_1,\ldots,\phi_k)
	=
\text{\rm mult}(\phi_1,\ldots,\phi_k, L_{k+1},\ldots, L_n).
\end{equation*}
By the semicontinuity of multiplicity, 
there exists a neighborhood $U$ of $0$ such that the multiplicity of 
$(\phi_1,\ldots,\phi_k)$ at $p\in U\cap\{\phi=0\}$ is less or equal to $\mu$.
After shrinking $U$ if necessary, we may assume that there exist finite number of generators $h_1,\ldots,h_m$ of $\sqrt{(\phi_1,\ldots,\phi_k)}$
such that each $h_j$ is holomorphic on $U$.

Let $p\in U\cap\{\phi=0\}$. 
Choose $k$-dimensional generic linear subspace passing through $p$ defined by $L^p_{k+1},\ldots,L_n^p$ such that
$$
	\mult(\phi_1,\ldots,\phi_k,L_{k+1}^p,\ldots, L_n^p)\leq \mu.
$$
Choose generic linear combinations  $\psi_1,\ldots, \psi_{k-1}$ of components of $\phi$ and define irreducible curves 
$C_1,\ldots,C_\delta$ by
$$
    \bigcup_j C_j=\mathcal{V}(\psi_1,\ldots,\psi_{k-1},L_{k+1}^p,\ldots,L_n^p).
$$
Let $\gamma_j:(\Delta,0) \to (C_j,p)$ be a
local parametrization of $C_j$ at $p\in C_j\cap \{\phi=0\}$.
Since 
$$
	h^\mu\in (\phi_1,\ldots,\phi_k, L_{k+1}^p,\ldots,L_n^p)
$$ 
for any $h\in \mathcal{O}_{n,p}$ with $h(p)=0$ 
(see e.g.\ \cite[\S2.3.3]{D93},
\cite[Lemma I.5]{S10}),
we obtain for each $j=1,\ldots,\delta$, and $\ell=1,\ldots,m$,
$$
\mu \cdot \ord(h_\ell\circ \gamma_j)=
	\ord(h^\mu_\ell \circ \gamma_j)
    	\geq
\min_{\ell=1,\ldots,k} \ord(\phi_\ell\circ \gamma_j).
$$
Hence
$$
    \tau_p:=\max_j \frac{
    	\ord(\gamma_j^*(\phi))
	}{
	\ord(\gamma_j^*(\sqrt{(\phi)}))}\leq \mu,
$$
where $(\phi)$ denotes the generated ideal as before.

As in Heier~\cite{Heier}, define the type $\mathcal{T}(\phi)$ by
$$
   \mathcal{T}(\phi):=\sup_{\gamma\in \Gamma_0 }\left\{\frac{\ord(\gamma^*(\phi))}{\ord(\gamma^*(\sqrt{(\phi)}))}\right\},
$$
where $\Gamma_0$ is the set of germs of local holomorphic curves $\gamma:(\Delta,0)\to (\mathbb{C}^n,p)$ with $p$ in the zero set of $(\phi)$  whose image is not contained in
$\mathcal{V}(\phi).$
We claim that for sufficiently small $U$, 
\beq\Label{ttp}
   \sup_{p\in U}\tau_p=\mathcal{T}(\phi)
\eeq
for a generic choice of $\psi=(\psi_1,\ldots, \psi_{k-1})$ and $L_{k+1}^p,\ldots,L_n^p$. 
Then the desired statement follows from  \cite[Theorem 3.2]{Heier}.

To prove the claim \eqref{ttp}, we follow \cite[Section 3.2]{Heier} by letting
$$
	Bl_\phi(U) \to U
$$
be the blowing up of the ideal $(\phi)$ on an open neighborhood $U$ of
$0\in \mathbb{C}^n$ and let $X^+ \to Bl_\phi(U)$ be the normalization of the blowing up.
Then there exists an effective Cartier divisor $F$ such that
$$
   (\phi) \cdot \mathcal{O}_{X^+}=\mathcal{O}_{X^+}(-F),
$$
where $(\phi)$ is identified with its pullback.
Write
$$
F=\sum _{i=1}^s r_i E_i.
$$
After shrinking $U$ if necessary, we may assume that 
$$\bigcap_i \sigma(E_i)\ni 0,$$
where $\sigma:X^+\to U$ is the natural map given by blowing up and normalization.
Let 
$$
	m_i:=\ord_{E_i}(h_1\circ \sigma,\ldots,  h_m\circ \sigma),
	~i=1,\ldots,s
$$
be the vanishing order of the ideal 
$(h_1\circ \sigma,\ldots,h_m\circ \sigma)$ at 
generic points of $E_i$.
By Proposition 3.4 of \cite{Heier},
\begin{equation}\label{heier-t}
	\mathcal{T}(\phi)= \max_i \left\{\frac{r_i}{m_i}\right\}.
\end{equation}
Then claim can be proved 
by following the arguments of the proof of Theorem 3.5 of \cite{Heier}.
For reader's convenience, we include details as follows.

Assume that 
$$
	\frac{r_1}{m_1}=\max\left\{\frac{r_i}{m_i}\right\}.
$$
Choose a function $g$ whose divisor is $E_1$:
$$
	E_1={\rm div}(g).
$$
Then for generic $\tilde p\in E_1$ and a curve $\tilde\gamma$ whose image is not contained in $E_1$ and $\tilde\gamma(0)=\tilde p$, we obtain
\begin{equation}\label{v-ord}
	\frac{\ord(\tilde\gamma^*(\phi\circ\sigma))}
	{ \ord(\tilde\gamma^*(h\circ\sigma))}
	=\frac{\ord(\tilde\gamma^* (g^{ r_1}))}
	{\ord(\tilde\gamma^*(g^{ m_1}))}
	=\frac{r_1}{m_1}.
\end{equation}

Let $p=\sigma(\tilde p)\in\{\phi=0\}\cap U$ and let $\gamma_j\colon (\mathbb{C},0)\to (C_j, p)$, $j=1,\ldots,\delta$ be as before.   Let
$$
	D_j={\rm div}(\psi_j),~\sigma^*(D_j)=F+\widetilde D_j,\quad j=1,\ldots,k-1
$$
and let
$$
	G_j={\rm div}(L_{j}^p),~\sigma^*(G_j)=\widetilde G_j,\quad j=k+1,\ldots,n.
$$
Then
$$	
	\widetilde D_1\cap\cdots\cap \widetilde D_{k-1}\cap \widetilde G_{k+1}\cap\cdots\cap \widetilde G_n
$$
is a union of curves. Since normalization $X^+ \to Bl_\phi(U)$ is a finite map, by Bertini's Theorem, for generic choice of $\tilde p\in E_1$, $\psi_1,\ldots,\psi_{k-1}$ and $L^p_{k+1},\ldots,L_n^p$, one of the curves, say $\tilde \gamma_1$, the lift of $\gamma_1$,
is a smooth curve meeting $ E_1$ at some generic point $\tilde\gamma_1(0)\in \sigma^{-1}(p)$. 
Since $\gamma_1=\sigma\circ \tilde \gamma_1$ , we obtain
$$
	\frac{\ord(\gamma^*_1(\phi))}
	{\ord(\gamma_1^*(h))}
	=\frac{\ord(\tilde\gamma^*_1(\phi\circ\sigma))}
	{\ord(\tilde \gamma_1^*(h\circ\sigma))}
	=\frac{r_1}{m_1}
$$
achieving the maximum in \eqref{heier-t}
proving \eqref{ttp} as desired.
\end{proof}

\section{Multiplicity estimates for Jacobian determinants}


We have the following more precise version
of the meta-procedure (MP1) in Theorem~\ref{meta}
that can be of independent interest:

\bp[Selection of a partial Jacobian determinant]\Label{v}
Let
$$(f,\psi)\colon (\C^{m+d},0) \to (\C^m\times\C^d ,0),
    \quad
    m, d\ge 1,
 $$
be a finite holomorphic map germ.
Then after a linear change of holomorphic coordinates
$(z_1,\ldots,z_{m+d})$ around $0$, one has
\beq\Label{min}
         \mult\left(
         	f, 
		\frac{
			\partial(\psi_1,\ldots,\psi_d)
		}{
			\partial (z_{1},\ldots,z_{d})
		}
	\right)
         	\le 
	d\cdot \mult f \cdot \mult(f,\psi).
\eeq
\ep

\bpf
Let
$$
    V:=\{f=0\}\subset \mathbb{C}^{m+d}
$$
be the germ at $0$ of a subvariety defined by $f$.
By additivity of the intersection number, we may assume that $V$ is irreducible.
By Theorem 2.0.2 of \cite{W}, 
we can choose an embedded desingularization
$r\colon \2M\to U$, where $U$ is a neighborhood of $0$ in $\C^{m+d}$,
with exceptional divisor $E\subset \2M$,
such that the strict transform $\2V\subset \2M$ of $V$ is regular.
We write 
$\sigma\colon\widetilde V\to V$ for the restriction of $r$. 
Since $V$ is of pure dimension $d$, $\widetilde V$ is a smooth variety of dimension $d$.
Let
$$
	E=\sum_j c_j E_j,
$$
where $E_j$ are irreducible. Note that $\widetilde V$ has only simple normal crossings with $E$.
Since the desingularization $r$ of \cite{W} is given by a sequence of blow-ups at smooth center which has simple normal crossings with exceptional divisors of the previous blow-up, 
we may assume $E_j$ is smooth. 
Note further that the blow-up centers of $\sigma$ are disjoint from the set $V_{reg}$ of points where $V$ is smooth.

Let 
$$
	\mu:=\mult(f,\psi)
$$
and
$$
	\Psi:=\psi\circ \sigma\colon \widetilde V \to \C^d.
$$
Choose open neighborhoods $W_j\subset \widetilde V$ of $E_j\cap \widetilde V$. Then
$\cup_j \sigma(W_j)$ is an open neighborhood of $0$ in $V$ and therefore, $\cup_j \Psi(W_j)$ is an open neighborhood of $0\in \mathbb{C}^d$. Since there are only finite number of divisors $E_j$, there exists $j_0$ such that $\Psi(W_{j_0})$ contains an open sector attached to $0$.
Since each blow-ups has exceptional divisors of the form $M\times C$, where $C$ is a smooth center and $M$ is a compact manifold,  
we can cover $E_{j_0}$ with finite number of open sets of the polydisc form and at least one of them maps to a set containing open sector attached to $0\in \mathbb{C}^d$ via $\Psi.$

Consider $\Psi^{-1}(0). $
Then $\Psi^{-1}(0)$ is a subvariety in $\widetilde V$ of codimension $k\leq d$. Choose a regular point $p_0\in \Psi^{-1}(0)$ and a local coordinate system 
$(x,y)=(x_1,\ldots,x_k, y_{k+1},\ldots,y_{d})$ centered at $p_0=(0,\ldots,0)$  such that
$$
	\Psi^{-1}(0)=\{x_1=\cdots=x_k=0\}.
$$
Assume that the image of a neighborhood 
$\{(x,y): x\in U_1, y\in U_2\} $ of $p_0\in \widetilde V$ via $\Psi$ contains an open sector 
attached to $0$.
For $y\in U_2$, define
$$
	N_y=\{(x, y):x \in U_1\}\subset \widetilde V.
$$
Then for generic $y$, the map
$\sigma_y:=\sigma_{N_y}:N_{y}\to V$
is one to one
and 
$\Psi_y:=\Psi_{N_y}:N_y\to \Psi(N_y)$ is a finite branched holomorphic covering of sheet number $\leq \mu.$  

First we will show that
\beq\Label{finite}
	\dim \mathcal{O}_{k}/(\Psi_y)\leq \mu,
\eeq
where $(\Psi_y)$ is an ideal generated by the components of $\Psi_y$. It is enough to show that
$$
	\dim \mathcal{O}_{k}/(\Psi_0)\leq \mu.
$$
Choose a regular curve $\gamma:(\mathbb{C},0)\to (N_0, 0)$. Suppose
$$
	\dim \mathcal{O}_1/\gamma^*(\Psi_0) >\mu.
$$
Then the equation $\Psi_0(\gamma(\zeta))=\Psi_0(\gamma(a))$ has at least $\mu+1$ solutions for generic $a\in \mathbb{C}$ sufficiently close to $0$,
which contradicts the assumption that $\Psi_0$ is at most $\mu$ to one. 
Therefore \eqref{finite} holds.
In particular, 
$$
	x_i^\mu\in (\Psi_y),
	\quad
	i=1,\ldots,k,
$$
which implies
$$
   \|x\|^\mu \le c(y)\|\Psi_y(x))\|,\quad (x,y)\in N_y
$$
for some function $c(y)$. By considering the expression of $x_i^\mu$ in terms of the components of $\Psi_y$, we can choose a point $y_0$ sufficiently close to $0$ and a neighborhood $U_2'\subset U_2$ of $y_0$ such that
\beq\Label{q}
	   \|x\|^\mu \le c_0\|\Psi_y(x))\|,\quad (x,y)\in N_y,\quad y\in U'_2
\eeq
for some constant $c_0$.
	 
Let $J(x,y)$ be the Jacobian determinant of $\Psi.$ Since $J\equiv 0$ on $\Psi^{-1}(0)$, $J$ is in the ideal generated by $x=(x_1,\ldots,x_k)$.
Consider the power series expansion 
\beq\Label{J-exp}
	J(x,y)=\sum_\a c_\a(y)x^\a
\eeq
at $(0, y_0)$.

We will show:

{\bf Claim.} There exists $\a$ with $\|\a\|\leq d \mu$ such that $c_\a\not\equiv 0.$

Assuming the claim does not hold, after shrinking $U_1$ and $U'_2$ if necessary, we obtain 
\beq\Label{assumption}
	|J(x,y))\leq c_1 \|x\|^T,\quad (x,y)\in U_1\times U'_2
\eeq
for some $T> d\mu $, where the constant $c_1$ is independent of $x$ and $y$.
Choose a sector 
$$
	S\subset\Psi(U_1\times U_2')\backslash \Psi(\{ J= 0\}\cup E).
$$
Then 
\beq\Label{volume}
    |
    	J(x,y)
    |
    \le c_1\|x\|^{T},
    \quad
    x\in (\Psi)^{-1}(S).
\eeq
Let $\omega$ and $\widetilde \omega$ be standard volume forms on $\C^d$
and $\widetilde V$ respectively.
Then \eqref{volume} implies
$$
    (\Psi^*\omega)(x,y) \le c_2 \|x\|^{2T} \cdot \widetilde\omega,
    \quad
   (x,y)\in (\Psi)^{-1}(S).
$$
Write $B_{\eps}\subset \mathbb{C}^d$ for the ball of radius $\eps$ centered at $0$
and consider the sectorial region
$$
    D_\eps := S\cap (B_{2\eps}\setminus \1{B_\eps}) \subset \C^d,
$$
whose volume satisfies
\beq\Label{vol}
    {\rm volume}(D_\eps) = \int_{D_\eps} \omega \ge c_3 \eps^{2d}.
\eeq
Since the volume of $\delta$-ball in $U_1$
 does not exceed the volume of the $\delta$-ball in $\C^k$ up
 to a constant,
 we have
\beq\Label{del}
    \int_{\Psi^{-1}(D_\eps)} \Psi^*\omega \le c_4 \delta^{2T}
    \int_{B_\delta\times U_2'} \widetilde \omega
    \le c_5  \delta^{2T+2k},
\eeq
for all $\delta$ with $\Psi^{-1}(D_\eps)\subset B_\delta\times U'_2.$
In view of \eqref{q}, we can choose $\delta = c_6 \eps^{1/\mu}$ in \eqref{del}.
Then, together with \eqref{vol}, we obtain
$$
    \eps^{2d} \le c_8 \eps^{(2T+2k)/\mu}
$$
for any sufficiently small $\eps>0$, which is a contradiction proving the claim.

Since $J$ is holomorphic, there exist an integer $m\leq d\mu$ and a variety $ X$ such that for each $p\in \Psi^{-1}(0)\setminus X$, $\Psi^{-1}(0)$ is smooth at $p$ and there exists 
a holomorphic coordinates $(x,y)=(x_1,\ldots,x_k, y_{k+1},\ldots,y_d)$ at $p$ such that 
$$
	\Psi^{-1}(0)=\{x=0\}
$$
and
$$
	J(x,y)=x_1^m+\sum_{\ell<m}c_\ell(x', y)x_1^\ell
$$
where $x'=(x_2,\ldots,x_k)$.	 
Choose a point $p_1\in \Psi^{-1}(0)\setminus X$ and a generic smooth holomorphic curve $\gamma:(\mathbb{C},0)\to(\widetilde V, p_1)$
such that 
$$
	\ord(J\circ \gamma)=m.
$$
By chain rule, we obtain
$$
   \sigma^*( d\psi_1\wedge\cdots\wedge d\psi_d\big|_{V})=J dw_1\wedge\cdots\wedge dw_d,
$$
where $w=(w_1,\ldots,w_d)$ is a coordinate system of $\widetilde V$ centered at $p_1$. Therefore we obtain
$$
    \sum_{(i_1,\ldots,i_d)}\frac{\partial(\psi_1,\ldots,\psi_d)}{\partial (z_{i_1},\ldots,z_{i_d})}
    (\sigma\circ\gamma(\zeta))
    \frac{\partial(\Sigma_{i_1},\ldots,\Sigma_{i_d})}{\partial(w_1,\ldots,w_d)}(\gamma(\zeta))
    =
    J(\gamma(\zeta)),
$$
where $\sigma=(\Sigma_1,\ldots,\Sigma_{m+d})$.
Since $\sigma$ is one to one outside the exceptional divisor of $\sigma$, we can choose $(i_1,\ldots,i_d)$ such that
$$
    \frac{\partial(\Sigma_{i_1},\ldots,\Sigma_{i_d})}{\partial(w_1,\ldots,w_d)}(\gamma)\not\equiv 0
$$
and
\beq\Label{local-coordinates}
    \ord\left(\frac{\partial(\psi_1,\ldots,\psi_d)}{\partial (z_{i_1},\ldots,z_{i_d})}
    (\sigma\circ\gamma)\right)
    \leq d\mu.
\eeq

Let 
$$
	\nu:=\mult f.
$$
Consider $(m+1)$-dimensional submanifold
$L\subset \mathbb{C}^{m+d}$ passing through $0$
such that 
$\mult f$ equals the multiplicity of $f|_L$ at $0$. 
Since the singular locus of $V$ is of codimension $\geq m+1$ in $\C^{m+d}$, 
we may assume
$$
	L\cap V\subset V_{reg}\cup \{0\}
$$ 
and therefore $\widetilde V\cap \widetilde L\cap E$ is a finite set, where $\widetilde L$ is the strict transform of $L$ under $\sigma$. In fact, 
$\widetilde V\cap \widetilde L\cap E$ has at most $\nu$ elements.
Since all smooth centers of $\sigma$ are disjoint from $V_{reg}$, we may further assume that 
$L$ intersects smooth centers only at $0$
and only transversally and hence each connected component of $\widetilde L$ is smooth.

We claim that for $L$ with generic tangent space $T_0L$, 
$\widetilde V\cap\widetilde L$ is a disjoint union of at most 
$\nu$ smooth curves that intersect transversally with $\widetilde V\cap E$.
Indeed,
let $A$ be the space of all $\nu$-jets at $0$ of all holomorphic curve germs $\gamma\colon(\C, 0)\to ( \2M, \gamma(0))$ such that $\gamma(0) \in \widetilde V\cap E$. 
Then $A$ is a subvariety in the jet space of all $\nu$-jets of holomorphic curves and 
the subset $B$ of all $\nu$-jets of curves $\gamma$ which are
either singular or intersecting tangentially with $\widetilde V\cap E$ at $\gamma(0)$ is a proper subvariety.
Let $p_0$ be a smooth point of $E\cap \widetilde V$. Choose a germ of a smooth curve $\gamma:(\mathbb{C},0)\to (\widetilde V, p_0)$ which intersects transversally with
$\widetilde V\cap E$ at $p_0$, 
i.e.\ $j^\nu_0\gamma\ \in A\setminus B$.
 Since $\sigma\circ \gamma$ is a holomorphic curve in $V$ passing through $0$, 
we can choose a set of holomorphic functions 
$\phi_j$, $j=1,\ldots,d-1$, such that
$\{f=\phi=0\}$ is $1$-dimensional and 
$$
	{\rm image}(\sigma\circ\gamma)\subset \{f=\phi=0\},
$$ 
where $\phi=(\phi_1,\ldots,\phi_{d-1})$. 
Then after a small perturbation of the differential of $\phi$,
we may assume that 
$L=\{\phi =0\}$ 
is smooth and $\{f=\phi=0\}$ 
contains a curve $\sigma\circ \gamma$ with 
\ $j^\nu_0\gamma\ \in A\setminus B$,
and all other components are $1$-dimensional.
Since $V$ is irreducible, any other component of $L\cap V$
can be connected with $\sigma\circ\gamma$ via paths in $V$.
Hence any two components can be connected
via a deformation of $\phi$ as regular function of its differential at $0$
that can be used as parameter,
proving the claim.

Let $X$ be as before. Then $X$ is a proper subvariety of $\widetilde V\cap E$. Let 
$C$ be the subset of $A$ consisting of $\nu$-jets of curves $\gamma$ such that $\gamma(0)\in X$. Then $C$ is a proper subvariety of $A$. 
Since the union of all linear subspaces covers $\mathbb{C}^{m+d}$, there exists $L$ such that the $\nu$-jet at $\sigma^{-1}(0)$ of the 
lifting of a curve in $L\cap V$ under $\sigma$ is not contained in $C$. Let
$\nu_0\leq \nu$ be the number of irreducible components of $V\cap L$ for generic linear subspace of codimension $d-1$.
Then for a generic linear subspace $L$ of codimension $d-1$ we can define 
a map $j^\nu_0:\{L\}\to A\times\cdots\times A$ by
$$
	j^\nu_0(L):=(j_{0}^\nu \gamma_1,\ldots,j^\nu_{0}\gamma_{\nu_0}),
$$
where 
$\cup_j\gamma_j$ is the lifting of $V\cap L$ under $\sigma$ such that
$$
	\sigma(\gamma_\ell(0))=0,
	\quad \ell=1,\ldots,\nu_0.
$$ 
Since $B$ and $C$ are proper subvarieties of $A$, there exists $L$ such that 
\beq\Label{generic-L}
	j^\nu_{0} \gamma_\ell \in A\setminus(B\cup C),
	\quad \ell=1,\ldots,\nu_0.
\eeq
Therefore by the additivity of multiplicity, 
we obtain
$$
	\mult\left(f\big|_L, \frac{\partial(\psi_1,\ldots,\psi_d)}{\partial (z_{1},\ldots,z_{d})}\big|_L\right)
        \leq d\nu\mu
$$
for some coordinates at $0$ constructed by a linear change of local coordinate systems satisfying \eqref{local-coordinates} associated to irreducible components of $V\cap L$ .
\epf

In the following corollary
we use the convention that
$\mult (f)=1$ if $f$ has $0$ components.

\bc\Label{multiplicity-J}
Let
$(f,\psi)\colon (\C^n,0)\to (\C^{n-d}\times\C^d,0)$ 
be a holomorphic map germ satisfying
$$
	\nu:=\mult (f) <\infty,
	\quad 	
	\mu:= \mult (f, \psi)  <\infty.
$$
Then after a linear change of 
coordinates 
$(z_1,\ldots,z_{n})$
and another linear coordinate change in $\C^d$,
the partial Jacobian determinant
\beq\Label{J-part}
	J:=\frac{
		\partial(\psi_1,\ldots,\psi_{d})
	}{
		\partial(z_1,\ldots,z_{d})
	}
\eeq
satisfies
\beq\Label{fjpsi}
    \mult(f, J)\le d \nu\mu,
    \quad
    \mult(f, J, \psi_2,\ldots,\psi_d)\le d \nu\mu^{d},
\eeq
where $\psi_j$ is the $j$-th component of $\psi$ in the new coordinates.
\ec

\bpf
Applying Proposition~\ref{v} to $\psi$,
we conclude that the first inequality
in \eqref{fjpsi} holds
after a generic linear change of $z$.
The second inequality is obtained by applying Lemma~\ref{siu}
to $(f,J)$ and $F:=(f,\psi)$.
%
\epf

\section{Existence of effective triangular resolutions}

The following is a more precise version
of the meta-procedure (MP2) in Theorem~\ref{meta}:

\bp\Label{triangle}
Let $1\le k\le n$, 
 $f_1, \ldots, f_k, \phi_1,\ldots, \phi_n \in \6O_{n,0}$, 
satisfy
\beq\Label{muj}
	\mu_j:=
	\mult(f_1, \ldots, f_j, \phi_{j+1}, \ldots, \phi_n) <\infty,
	\quad
	1\le j\le k.
\eeq
Then
 the map germ $\phi=(\phi_1,\ldots, \phi_n)$ 
and the filtration of ideals
$I_j:=(f_1,\ldots, f_j)$, $1\le j\le k$,
admit a triangular resolution 
$h=(h_1,\ldots, h_k)$
satisfying
\beq\Label{h-mult}
	\ord_{w_j} h_j \le 
	n\cdot \mu_j \cdot \mult(f_1,\ldots, f_j) 
	,
	\quad
	1\le j\le k.
\eeq
In fact, each $h_j(w_j, \ldots, w_n)$ can be chosen
as Weierstrass polynomial in $w_j$.
\ep

\bpf
Consider 
the coordinate projections
$$
    \pi_j(w_1,\ldots, w_n)
    = 
    (w_{j}, \ldots, w_n)
    \in \C^{n-j+1}
    ,
    \quad
    1\le j \le k,
$$
and let
$$
	W_j := \6V \left(f_1,\ldots,f_j\right),
	\quad
	\widetilde W_j
	:=
	(\pi_{j} \circ \phi)(W_j)
	\subset\mathbb{C}^{n-j+1},\quad 1\leq j\leq k,
$$
where $\6V$ is the zero variety.
Then $W_j$ is of codimension $\ge k$ in $\cn$.
In fact, counting preimages and using \eqref{muj}, we 
conclude that
$\widetilde W_{j}\subset \mathbb{C}^{n-j+1}$ is a proper subvariety of codimension $1$
and
$$
	\pi_{j+1} |_{\widetilde W_j} \colon \widetilde W_j \to \mathbb{C}^{n-j}
$$
is a finite 
holomorphic 
map germ
of degree
$\le \mu_j$.
Then there exist  Weierstrass polynomials 
$Q_j(w_j, \ldots, w_n)$, $j=1,\ldots,k$, 
satisfying
$$
	Q_j=
	w_{j}^{\nu_j}+\sum_{\ell<\nu_j}b_{j,\ell}(w_{j+1},\ldots,w_{n})w_{j}^\ell,
	\quad
	Q_j|_{\2W_j} =0,
	\quad
	\nu_j
	=
	\ord_{w_j} Q_j
	\le \mu_j.
%
$$
Furthermore, Lemma~\ref{radical} implies
$$
	h_j\circ \phi
	\in(f_1,\ldots,f_j),
	\quad
	h_j:= Q_j^{\l_j},
$$
for suitable $\l_j\in \N_{\ge1}$ satisfying
$$
	\l_j \le n\cdot \mult(f_1,\ldots, f_j).
$$
Then $(h_1,\ldots,h_k)$ is a triangular resolution
satisfying \eqref{h-mult}
 as desired.
\end{proof}

\section{Effective Kohn's procedures for triangular resolutions}

The following is a more precise version
of the meta-procedure (MP3) in Theorem~\ref{meta}:

\bp\Label{claim}
Let $1\le k < n$ and $(Q_1,\ldots,Q_{k+1})$ be a triangular resolution 
of $(\G,\6I)$,
where 
$\G\colon (\cn,0)\to(\C^n,0)$
is a holomorphic map germ
and $\6I$
a filtration of ideals $I_1\subset \ldots \subset I_{k+1}\subset\6O_{n,0}$.
Assume
\beq\Label{mus}
	\mu_j = \ord_{w_j} Q_j <\infty,
	\quad
	1\le j\le k,
\eeq
and
\beq\Label{det-inc}
	I_{k+1}\subset  I_k+(J),
\eeq
where $J$ is the Jacobian determinant of $\G$.

Then 
$Q_{k+1}\circ \G$ can be obtained
by applying holomorphic Kohn's procedures 
(P1) and (P2)
to $(\psi, I_k)$
where 
$\psi=(\G_{k+1},\ldots,\G_n)$ 
are the last $k$ components
and
each procedure (P1) and (P2) is applied $\mu_1\cdots \mu_k$ number of times with the root order in (P2) being $\le k+1$.
In particular, if $I_k$ consists of multipliers of order $\ge \eps$,
then $Q_{k+1}\circ \G$ is a multiplier of order $\geq (2k+2)^{-\mu_1\cdots\mu_k} \eps$.
\ep

\begin{proof}
The proof is inspired by  \cite[III.7]{S10}, \cite[\S3]{S17} 
and \cite[\S5]{CD10}.
Since $\mu_j<\infty$ for $j\le k$, 
multiplying by invertible holomorphic functions,
we may assume that 
$$
	Q_j=
	w_{j}^{\mu_j}
	+\sum_{\ell<\mu_j}b_{j,\ell}(w_{j+1},\ldots,w_{n})w_{j}^\ell,
	\quad
	1\le j\le k,
$$
are Weierstrass polynomials satisfying
$$
	f_j:= Q_j\circ \G\in I_j.
$$
In addition, \eqref{det-inc} implies
\beq\Label{f-inc}
	f_{k+1} 
	:=
	Q_{k+1}\circ\G \in I_k + (J).
\eeq

We use the reverse lexicographic order for $\mathbb{N}^k$, i.e. for $L=(\ell_1,\ldots,\ell_k)$ and $\widetilde L=(\tilde\ell_1,\ldots,\tilde\ell_k)$, we let $L<\widetilde L$ if there exists $j_0$ such that $\ell_j=\tilde\ell_j$ for $j>j_0$ and $\ell_{j_0}<\tilde\ell_{j_0}.$
Let
$$
	\Lambda:=
	\{L=(\ell_1,\ldots,\ell_k)\in \mathbb{Z}^k :
	1\leq \ell_j\leq \mu_j,~j=1,\ldots,k\}.$$
For $L\in \Lambda,$ define $A_L\in \6O_{n,0}$ by
$$
	A_L(w):=\partial_{w_1}^{\ell_1}Q_1(w) \cdots\partial_{w_k}^{\ell_k}Q_k(w).
$$
Since $f_j\in I_j$ by the definition of a triangular resolution,
$$
	J_{(1,\ldots,1)}:=
	\frac{
		\partial(f_1,\ldots, f_k, \psi_{k+1},\ldots,\psi_{n})
	}{
		\partial(z_1,\ldots,z_{n}),
	}
$$
which is the Jacobian determinant of the map
$$
	(f_1, \ldots, f_k,  \psi_{k+1},\ldots,\psi_{n} ) 
	= \Phi \circ \G,
$$
where
\beq\Label{P}
	\Phi(w) 
		:= 
	(Q_1(w), \ldots, Q_k(w), w_{k+1}, \ldots, w_n),
\eeq
Then the Jacobian factors as 
$$
	J_{(1,\ldots,1)}
	=
	(A_{(1,\ldots,1)}\circ \G) J
$$
and hence by \eqref{f-inc},
$$
	 (A_{(1,\ldots,1)} \circ \Gamma)f_{k+1} 
	 \in I_k + (J_{(1,\ldots,1)})
$$
is obtained by applying Kohn's procedure (1)
to $I_k$.

Now for a given $L=(\ell_1,\ldots,\ell_k)\in \Lambda$ such that $L>(1,\ldots,1)$, assume by induction that
 that $(A_{\widetilde L} \circ \Gamma)f_{k+1}$ is 
 obtained by applying Kohn's procedures (P1) and (P2).
 Let
$$
	m_L := \# \{ j : \ell_j > 1\}  \le k
$$
be the number of $j\in \{1, \ldots, k\}$ with $\ell_j>1$.
We will show that $(A_L \circ \Gamma)f_{k+1}$ 
is obtained by additionally applying
once each of
the Kohn's procedures (P1) and (P2)
with the root order in (P2)
being $\le m_L+1\le k+1$. 
For $j=1,\ldots,k$, we define $B_j$ as follows:
\ben
\item
In the case $\ell_j>1$,
set
$$
	B_j:=(A_{L_j} \circ \Gamma)f_{k+1} 
	= 
	(A_{L_j} Q_{k+1}) \circ \Gamma
$$
with
$$
	L_j := (\mu_1,\ldots,\mu_{j-1},\ell_{j}-1,\ell_{j+1},\ldots,\ell_k) < L.
$$
%
\item
In the remaining case $\ell_j=1$, set
$$
	B_j:=Q_j \circ \Gamma =f_j.
$$
\een
By our assumption, $B_j$ in both cases
are obtained by applying Kohn's procedures.

Now apply (1) to obtain
$$
	J_L:=
	\frac{
		\partial(B_1,\ldots, B_k, \G_{k+1},\ldots,\G_{n})
	}{
		\partial(z_1,\ldots,z_{n})
	}.
$$
Then $J_L$ is the Jacobian determinant of the map
$$
	(B_1, \ldots, B_k, \G_{k+1}, \ldots, \G_n ) = \Phi \circ \G,
$$
with $\Phi$ as in \eqref{P}.
In view of our assumption that each $Q_j$, $j\le k$ 
is a Weierstrass polynomial in $w_j$ of degree $\mu_j$,
each top derivative $\d_{w_j}^{\mu_j}Q_j$ is constant
and hence $B_j$ only depends on $(w_j, \ldots, w_n)$.
Then using factorization of the Jacobian determinant
 and the triangular property of $B_j$'s,  
we obtain
$$
	J_L=
	c\left(
		\big(
		(\partial_{w_1}^{\ell_1}Q_1)^{m_1}
            \cdots(\partial_{w_k}^{\ell_k}Q_k)^{m_k}
		Q_{k+1}^{m_L}
		\big)
		  \circ
        \Gamma
        	\right)
	J
$$
for some constant $c\ne0$ and integers $m_j\leq m_L + 1$,
and hence by \eqref{f-inc},
$$
	 \left(
	 	(A_L Q_{k+1}) \circ \Gamma
	\right)^{m_L+1}
	 \in I_k + (J_L).
$$
Then  $(A_L Q_{k+1})\circ\G$ is obtained by the Kohn's procedure (2)
with root order $\le m_L+1\le k+1$,
and the proof is
complete by induction.
\end{proof}

%
%
%
%

\section{Effective estimates}

\bc\Label{iter}
Let
$0\le k < n\le N$ and
 $F_1, \ldots, F_N, f_1, \ldots, f_k\in \6O_{n,0}$, 
satisfy
\beq\Label{muj}
	\nu:= \mult(F_1,\ldots, F_N) < \infty,
		\quad
	\mu:=
	\mult(f_1, \ldots, f_k) <\infty,
\eeq
(where $\mu:=1$ in case $k=0$).
Then there exists a $(k+1)$-tuple 
$g_1,\ldots,g_{k+1}\in \6O_{n,0}$
with
\beq\Label{iter-mult}
	\mult(g_1,\ldots,g_{k+1})
	\le
		n^{k+3} \mu^{n+k+3} \nu^{(n-k)(n+1)},
\eeq
where $g_j\in (f_1,\ldots, f_j)$ for $1\le j\le k$,
and $g_{k+1}$ is 
obtained from the initial set $\{F_1,\ldots,F_N\}$
and the ideal $(f_1,\ldots,f_k)$
via the Kohn's procedures 
(P1) and (P2),
each applied 
$\le	(n\mu^2 \nu^{n-k})^k$
number of times 
with the root order in (P2) being $\le k+1$.
In particular, 
if $f_1,\ldots, f_k$ are multipliers of order at least 
$\eps>0$,
then $g_{k+1}$ is a multiplier of order 
\beq\Label{g-ord}
	\ge \frac{
		\eps
	}{
		(2k+2)^{(n\mu^2 \nu^{n-k})^k}
	}
	.
\eeq

Furthermore, 
after a linear change of coordinates 
$z=(z_1,\ldots,z_n)$,
$g_j$ can be chosen of the form
$$
	g_j(z)
	=
	Q_j(z_j,\ldots, z_k, \psi_{k+1} (z) ,\ldots,\psi_{n} (z) ),
	\quad
	1\le j\le k+1,
$$
where 
$\psi_j$'s are linear combinations of $F_j$'s.
%
\ec

\bpf
Applying Lemma~\ref{siu},
we obtain (generic)
linear combinations
$\psi=(\psi_{k+1},\ldots,\psi_{n})$
 of $F_j$'s
such that
\beq
	\mult(z_1,\ldots, z_k,\psi)
		\le
	\nu^{n-k},
	\quad
	\mult(f, \psi) 
		\le 
	\mu \nu^{n-k}.
\eeq
Next apply 
Corollary~\ref{multiplicity-J}
to obtain,
after (generic) linear change of coordinates
$z=(z_1,\ldots,z_n)$ 
and (generic) linear transformation of $\psi_j$'s,
\beq\Label{min'}
         \mult\left(
         	f, J
	\right)
         	\le 
	n\cdot \mult(f) \cdot \mult(f,\psi)
	\le
	n\mu^2\nu^{n-k},
	\quad
\eeq
and
\beq\Label{fjpsi}
	\mult(f, J, \psi_{k+2}, \ldots, \psi_n)
		\le
	n\cdot \mult(f) \cdot ( \mult(f,\psi) )^{n-k}
		\le
	n \cdot \mu (\mu \nu^{n-k})^{n-k},
\eeq
where
$$
	J:= 
		\frac{
			\partial(\psi_{k+1},\ldots,\psi_{n})
		}{
			\partial (z_{k+1},\ldots,z_{n})
		}
	.
$$

Now apply Proposition~\ref{triangle}
for the map germ
$$
	\G(z):=(z_1,\ldots, z_k,\psi(z)),
$$ 
and the filtration of ideals 
$$
	I_j:=(f_1,\ldots, f_j), 
		\; 
	1\le j\le k,
		\quad
	I_{k+1} := (f_1,\ldots, f_k, J),
$$
to obtain a triangular resolution $(h_1,\ldots, h_{k+1})$ 
satisfying
\beq\Label{qjs}
	\ord_{w_j} h_j
		\le
		n\cdot \mult(f) 
		\cdot
	\mult(f,\psi)
		\le
	n\mu^2 \nu^{n-k}
\eeq
for
	$1\le j\le k$
and
\begin{multline}
	\Label{ql1}
	\ord_{w_{j+1}} h_{k+1}
		\le 
		n\cdot \mult(f, J)
		\cdot
	\mult(f,J,\psi_{k+2},\ldots,\psi_{n})
	\\
		\le
	n\cdot 	
	n\mu^2\nu^{n-k}
	\cdot
	n \cdot \mu (\mu \nu^{n-k})^{n-k}
	\le n^3 \mu^{n-k+3} \nu^{(n-k)(n-k+1)}
	,
\end{multline}
in view of \eqref{min'} and  \eqref{fjpsi}.
Setting 
$$
	g_j := h_{j} \circ\G,
	\quad
	1\le j\le k+1,
$$
and counting preimages using the
multiplicativity of the multiplicity under composition, we obtain
$$
	\mult(g_1,\ldots,g_{k+1}) 
	\le 
	\ord_{w_1} h_1 \cdots \ord_{w_{k+1}} h_{k+1}
	\cdot 
	\mult(\G)
$$
	implying
$$
	\mult(g_1 ,\ldots, g_{k+1})
		\le
	(n\mu^2 \nu^{n-k})^k
	\cdot
		n^3 \mu^{3+n-k} \nu^{(n-k)(n-k+1)}
	\cdot \nu^{n-k}
	=
	n^{k+3} \mu^{n+k+3} \nu^{(n-k)(n+1)}
	.
$$

Since 
$J$ equals
the Jacobian determinant of $\G$
and
$I_{k+1}\subset  I_k + (J)$,
we can
apply Proposition~\ref{claim}
to conclude that $g_{k+1}$
can be obtained
by applying holomorphic Kohn's procedures 
to the initial set of components $\psi$ 
and the ideal  $I_k$
with each procedure (P1) and (P2) applied 
$$
	\ord_{w_1} h_1 \cdots \ord_{w_k} h_k
	\le
	(n\mu^2 \nu^{n-k})^k
$$ 
number of times with the root order in (2) being $\le k+1$.
%

In particular, if 
each $\psi_j$ is a pre-multiplier
and 
each $f_j$ a multiplier of order $\ge \eps$,
it follows that 
each
$g_j \in I_j = (f_1, \ldots, f_j)$
is a multiplier of order $\ge \eps$
for $1\le j\le k$
and
$g_{k+1}$
is a multiplier of order
\eqref{g-ord}
as desired.
\epf

\bc\Label{compute}
There exists a function $\eps\colon \N_{>0}\times \N_{>0} \to \R_{>0}$,
such that 
for every collection of
holomorphic pre-multipliers 
in $\C^n$
with multiplicity $\le \nu$,
there is a finite sequence of 
Kohn's procedures (1) and (2) in Definition~\ref{k-hol} 
producing the unit as a multiplier of order
$\ge \eps(n,\nu)$.
In fact, one can take
\beq\Label{eps-n}
\eps(n,\nu) =
	\frac1{
		4
		(2n+2)^{2(n\nu)^{ (3n)^{n+1} }}
	}.
\eeq
\ec

\bpf
The case $n=1$ is well-known: for any premultiplier $\psi$
of multiplicity $\le \nu$, the unit is a multiplier of order
$\ge 1/4\nu$, which is greater than $\eps$ given by \eqref{eps-n}
with $n=1$.

Thus we may assume $n\ge2$ in the following.
Applying Corollary~\ref{iter} repeatedly
starting from $k=0$,
 we 
obtain holomorphic function germs 
$$
	f_{k,j}\in \6O_{n,0},
	\quad
	1\le j\le k\le n,
$$
with 
\beq\Label{fkk}
	\mult(f_{k,1}, \ldots, f_{k,k}) \le \mu_k,
\eeq
where $\mu_k$ are given
recursively by
$$
	\mu_0=1,
	\quad
	\mu_{k+1}
	 	= 
	 n^{k+3}  \nu^{(n-k)(n+1)} \mu_k^{n+k+3},
$$
from which
$$
	\mu_k = n^{a_k} \nu^{b_k},
$$
where the integer sequences $(a_k)$ and $(b_k)$
are given recursively by
$$
	a_1=3,
	\quad
	b_1=n(n+1),
	\quad
	a_{k+1} = (n+k+3)a_k +  k+3,
	\quad
	b_{k+1} = (n+k+3)b_k + (n-k)(n+1). 
$$
Since $n\ge1$, it is easy to see by induction that
$
	a_k\ge k+3
$,
from which it follows that
$$
	a_{k+1} \le (n+k+3)a_k +  a_k = (n+k+4)a_k,
	\quad
	k\ge 1,
$$
and then by induction
$$
	a_{k+1} \le (n+5)(n+6)\cdots (n+k+4) a_1 \le 3(2n+3)^{k}
	\le (3n)^{k+1},
	\quad
	n\ge 3,
$$
while for $n=2$, the same inequality can be checked directly:
$$
	a_2 \le  3(2+5) \le (3\cdot 2)^2.
$$

Similarly,
$$
	b_{k+1} \le (n+k+3)b_k + b_k
	= (n+k+4)b_k,
$$
whence
$$
	b_{k+1} \le (n+5)\cdots (n+k+4)b_1
	\le n(n+1)(2n+3)^k
	\le (3n)^{k+2},
	\quad
	n\ge 3,
$$
and directly for $n=2$:
$$
	b_2 \le (2+5) \cdot 2(2+1) \le (3\cdot 2)^3.
$$
Putting all together, we obtain
\beq\Label{muk}
	\mu_k \le n^{(3n)^k} \nu^{(3n)^{k+1}}.
\eeq

Next starting with $\eps=1/2$
and iterating \eqref{g-ord}, we conclude
that $f_{k,j}$ are multipliers of order $\ge \eps_k$,
where
$$
	\eps_0 =1/2,
	\quad
	\eps_{k+1} = 
	 \frac{
		\eps_k
	}{
		(2k+2)^{(n\mu_k^2 \nu^{n-k})^k}
	}
	,
$$
hence
$$
	\eps_{k+1} \ge 
	\frac{1}{
	 2(2n+2)^{A_{k+1}}
	 },
$$
where
$$
	 A_{k+1} = \sum_{j=1}^k (n\nu^{n-j} \mu_j^2)^j
	 \le k (n \nu^{n-k} \mu_k^2)^k
	 \le (n-1) (n\nu^{n-1} \cdot n^{2(3n)^k} \nu^{2(3n)^{k+1}})^k
	 \le 
	 	(n\nu)^{(3n)^{k+2}},
$$
and finally
$$
	\eps_n \ge
	\frac1{
		2(2n+2)^{(n\nu)^{ (3n)^{n+1} }}
	}.
$$

Now consider the monomials
$$
	z_j^{\mu_n} \in (f_{n,1}, \ldots, f_{n,n})
$$
and apply Kohn's procedure (2) to conclude
that each $z_j$ is a multiplier of order
$
	\eps_n/\mu_n,
$
and applying procedure (1), we conclude that $1$
is a multiplier of order
$$
	\frac{\eps_n}{2\mu_n}
	\ge
	\frac1{
		4(2n+2)^{(n\nu)^{ (3n)^{n+1} }}
		(n \nu)^{(3n)^{n+1}}
	}
	\ge
	\frac1{
		4
		(2n+2)^{2(n\nu)^{ (3n)^{n+1} }}
	}
	.
$$
\epf

\section{Examples of the effective algorithm}\Label{example}
We shall consider here
classes of special domains in $\C^4$
locally given by
$$
	\Im z_4 + \sum_{j=1}^3 |\psi_j(z_1,z_2,z_3)|^2 < 0,
$$
%
where
\beq\Label{co-squares}
	\psi_j(z)=z_j^2+ \ldots,
	\quad
	j=1,2,3,
	\quad
	z\in \C^3,
\eeq
where we use the following convention:

{\bf Convention}.
The dots stand for arbitrary terms of order 
higher than all previous terms.

We shall perform our analysis in a neighborhood of $0$
and hence regard all functions as germs at $0$.
The type (at $0$) of $\psi=(\psi_1, \psi_2, \psi_3)$ is $2$ and the multiplicity
$$
	\mu :=	\mult(\psi_1, \psi_2, \psi_3) = 2^3=8.
$$

In course of our effective algorithmic process, 
we repeatedly apply linear changes of coordinates $z=(z_1, z_2, z_3)$,
while it will be convenient to rename and keep the old coordinates
as $u=(u_1, u_2, u_3)$.

\subsection*{Step 0: The first multiplier -- select a Jacobian determinant of pre-multipliers.}
In general we select pre-multipliers 
among arbitrary linear combinations of the $\psi_j$
with constant coefficients.
(Note that combination with holomorphic coefficients
may not be pre-multipliers).
In our case
the Jacobian ideal is generated by the single function
$$
	J_1(z) =  
	\frac{
    		\partial(\psi_{1}, \psi_2, \psi_{3})
	}
    	{
		\partial(z_{1}, z_2 ,z_3)
	}
	(z)
	\sim
	z_1 z_2 z_3 + \ldots,
	\quad
	\mult J_1 = 3,
$$
where we generally write $\sim$
for an equality up to an
invertible germ of holomorphic function.

\subsection*{Step 1.1: Select generic complementary pre-multipliers.}
Following Siu \cite[(III.3)]{S10}, we look for pre-multipliers
$\phi_2, \phi_3$ among generic linear combinations of $\psi_j$
complementing $J_1$ in the sense of 
an effective joint multiplicity, i.e.\
such that
$
	\mult(J_1, \phi_2,\phi_3)
$
is effectively bounded.
In our case, we can replace $\psi_j$ with combinations
$$
	\psi_j := z_j^2 - z_1^2 + \ldots,
	\quad
	j=2, 3,
$$
that yield the effective bound
\beq\Label{quo}
	\mult(J_1, \psi_2, \psi_3)  
	= \dim \6O_3 / (J_1, \psi_2, \psi_3)
	\le 5\cdot 2\cdot 2 =20.
\eeq
Indeed
$$
(z_2 z_3)J_1 \sim z_1z_2^2 z_3^2 + \ldots
= z_1^5 + \ldots 
\mod (\psi_2, \psi_3),
$$
and hence the quotient in \eqref{quo}
is spanned by the monomials 
$$
	z_1^{\a_1} z_2^{\a_2} z_3^{\a_3}, \quad \a_1\le 4, 
	\quad
	\a_2, \a_3 \le 1.
$$

\subsection*{Step 1.2: Select a direction of an effective vanishing order.}
We next look to change the linear coordinates $(z_1, z_2, z_3)$
to make effectively bounded the vanishing order
of $J$ along $z_2=z_3=0$ and
the multiplicity
$
	 \mult(z_1, \psi_2, \psi_3)
$.
For this, consider coordinate  change with
$$
	(u_1, u_2, u_3) = (z_1, z_2 + z_1, z_3 + z_1),
$$
so that our data take the form
\beq\Label{JJ}
	J \sim z_1(z_2 + z_1)(z_3 + z_1) + \ldots
	= u_1u_2u_3 +\ldots
\eeq
along with
$$
	\psi_j = (z_j + z_1)^2 - z_1^2 + \ldots
	=
	u_j^2 - u_1^2 +\ldots
$$
and
$$
	\mult(\G_1) = 2^2=4,
	\quad
	 \G_1(z): = (z_1, \psi_2(z), \psi_3(z)),
$$
where $\G_1$ is the 
finite map as in Corollary~\ref{k}
that is used to decompose multipliers.

\subsection*{
	Step 1.3: Select a decomposable multiplier 
	$f_1 = Q\circ \G_1$ in the radical of $(J)$
}
We look for
a multiplier $f_1= Q\circ \G_1$ 
as in Corollary~\ref{k},
where $Q$
is a holomorphic function with effectively bounded 
vanishing order $\mu_1$ in $z_1$:
$$
	\mu_1 \le n \mu^n = 3\cdot 8^3.
$$
Any method can be used to construct $f_1$ here.
One possible approach consists of 
taking $f_1$ to be the product of all transformations of $J$
under the action of the
deck transformation group of the finite covering map germ
$\G_1\colon(\C^n,0)\to(\C^n,0)$,
i.e.\ setting
$$
	f_1 := \prod_{g\in {\rm Deck}(\G_\psi)} J\circ g,
$$
where ${\rm Deck}(\G_1)$ denotes 
the group of all deck transformations,
i.e.\ germs of biholomorphic transformations 
$g\colon(\C^n,0)\to (\C^n,0)$ with $\G_1\circ g = \G_1$.

In our case,
the deck transformation group of $\G_1$
is generated by the involutions
changing the signs of some of $u_j$, $j=2,3$,
i.e.\ transformations having in $u$-coordinates the form
$$
	g(z) =
		(z_1, \eps_2(z_2+z_1) - z_1, \eps_3(z_n+z_1) - z_1)
		+\ldots,
	\quad
	\eps_2, \eps_3 \in \{ 1, -1\}.
$$
Then from \eqref{JJ} we obtain
$$
	f_1
	\sim 
	\big(
		z_1(z_2 + z_1) (z_3 + z_1)
	\big)^4 + \ldots
	= (u_1 u_2 u_3)^4 +\ldots
	,
$$
where
$$
	\mult(f_1, \psi_2, \psi_3) \le n\mu^{n+1}
	= 3 \cdot 8^4.
$$

\subsection*{Step 2.1: Select a complementary partial determinant.}
We next follow the lines of the proof of 
Corollary~\ref{iter}
with $k=1$ and look for a linear
change of coordinates such that,
for the
partial determinant
\beq\Label{J2}
	J_2 = 
	\frac{
    		\partial(\psi_{2},\psi_3)
	}
    	{
		\partial(z_{2},z_{3})
	},
\eeq
the multiplicity
\beq\Label{m3}
	\mult(f_1, J_2, \psi_3)
\eeq
is effectively bounded.
 In general, the existence of such coordinates
 follows from Corollary~\ref{multiplicity-J}.
 In our case, we can apply a coordinate change with
$$
 	(u_1, u_2, u_3)
	=
	(z_1+z_2+z_3, z_2 + z_1, z_3 + z_1)
$$
and hence
transforming our data into
$$
	\psi_j = (z_j+ z_1)^2 - (z_1+z_2 + z_3)^2 + \ldots,
$$
and 
$$
	f_1 \sim 
	\big(
		(z_1+z_2+ z_3)(z_1 + z_2) (z_1 + z_3)
	\big)^{4} + \ldots.	
$$
Then computing \eqref{J2}  yields:
$$
	J_2 = (z_2+z_1)(z_3+z_1) - (2z_1 + z_2 + z_3)(z_1+z_2+z_3)+\ldots,
$$
or equivalently
\beq\Label{J2y}
	J_2 = u_2u_3 - (u_2+u_3)u_1+\ldots.
\eeq

\subsection*{Step 2.2: Select complementary pre-multipliers.}
By now we have an effectively bounded multiplicity 
$\mult(f_1, J_2)$.
In order to effectively bound \eqref{m3},
we replace $\psi_3$ with 
a linear combination of $\psi_j$,
denoted again by $\psi_3$ by a slight abuse of notation,
 $$
 	\psi_3 :=  u_1^2 + u_2^2 + u_3^2 +\ldots.
$$
Indeed, in the $u$-coordinates we have
$$
	f_1 \sim (u_1u_2u_3)^4 + \ldots,
$$
and it is easy to verify that
the lowest order terms
of $(f_1, J_2, \psi_3)$
have no common zeroes other than $z=0$
and 
\beq\Label{m3'}
\mult(f_1, J_2) \le 
\mult(f_1, J_2, \psi_3) \le 
n (\mu_1 \mu)^n
= 3\cdot (8\mu_1)^3.
\eeq

\subsection*{Step 2.3: Select decomposable function
$H = Q_2\circ \G_1$ in the radical of the ideal $(f_1, J_2)$}
We look for a function $H= Q_2\circ \G_1$ as in 
Proposition~\ref{claim},
where $Q_2$ is a Weierstrass polynomial
of effectively bounded degree in $w_2$
with holomorphic coefficients in $w_3$.
For our purposes, it will suffice to select $Q_2(w_2,w_3)$
to be just holomorphic with an effectively bounded vanishing order in $w_2$
$$
	\ord_{w_2} Q_2 = \nu \le n(\mu\mu_1)^n.
$$
Then using Proposition~\ref{claim}, 
we conclude that $H$ is a multiplier of
an order effectively bounded from below.

\subsection*{Step 2.4: Select new decomposable multipliers}
Proceeding as in the proof of 
Corollary~\ref{iter},
we take 
$$
	\G_{2}(z) :=
	(z_1, z_2, \psi_3(z)),
$$
and look for the multipliers of the form
$$
	\2f_j = \2Q_j \circ \G_2,
	\quad
	j=1,2,
$$
such that 
\beq\Label{nu12}
	\nu_1=\mult(\2Q_1(w_1, w_2, w_3), w_2, w_3)
	\le \mu\mu_1
	\quad
	\nu_2 = \mult(\2Q_2(w_2, w_3), w_3) \le
	n(\mu\mu_1)^{n+1}.
\eeq

\subsection*{Step 3.1: Select complementary partial determinant}
We now repeat the process with $k=2$
in the proof of Corollary~\ref{iter}.
We have
$$
\mult (f)\le
\mult(f, \2\psi) \le \mu \nu_1\nu_2 
\le n \mu (\mu\mu_1)^{n+2},
$$
where
$
f=(\2f_1, \2f_2).
$
Taking again suitable linear change 
of coordinates $z$, we may assume
as in Corollary~\ref{multiplicity-J}
that  
$$
	\mult(f, J_3)\le \mult (f) \; \mult(f, \2\psi)
	\le
	(n\mu)^2 (\mu\mu_1)^{2n+4},
$$
where
$$
\quad
J_3 = \frac{\d\psi_3}{\d z_3}.
$$

\subsection*{Step 3.2: Select a decomposable function in the radical of the ideal $(f, J_3)$}
Note that there is no analogue for Step 2.2 
because $(f, J_3)$ is finite.
Next, similarly to Step 2.3, 
select $\2H=\2Q_3\circ \G_2$,
where $\2Q_3=\2Q_3(w_3)$ has multiplicity effectively
bounded 
by
$$
	\ord_{w_3}\2Q_3 
	\le n(\mu \nu_1\nu_2)^n
	\le n (n\mu (\mu\mu_1)^{n+2})^n
	= n^{n+1}\mu^n (\mu\mu_1)^{n^2+2n}.
$$
Then $\2H$ is multiplier with effectively bounded order
by Proposition~\ref{claim}.

\subsection*{Step 3.3: Effective termination}
We obtain the multipliers
$$
	(\2f_1, \2f_2, \2H) = (\2Q_1, \2Q_2, \2Q_3)
	\circ \G_2
$$
and 
$$
	\mult(	\2f_1, \2f_2, \2H)
	\le \mult(\2Q_1, \2Q_2, \2Q_3)
	\cdot \mult(\G_2)
	\le 
	\nu_1\nu_2 \; (\ord_{w_3}\2Q_3) \cdot \mu.
$$
Thus we have constructed a triple
of multipliers of effective subellipticity orders with 
finite effectively bounded multiplicity.
Hence the linear coordinate functions $z_j$
are in the effective radical of 
$(\2f_1, \2f_2, \2H)$
and hence the algorithm terminates
by taking the Jacobian of $(z_1, z_2, z_3)$,
which is a multiplier
with an effectively bounded subellipticity 
order.
All effective bounds are directly computable
from the above estimates
similarly to Corollary~\ref{compute}.

\end{document}